\documentclass[12pt]{amsart}
\usepackage{amssymb}
\usepackage[all]{xy}

\newtheorem{theorem}{Theorem}[section]
\newtheorem{definition}[theorem]{Definition}
\newtheorem{proposition}[theorem]{Proposition}
\newtheorem{conjecture}[theorem]{Conjecture}

\begin{document}

\title{Flat matrix models for quantum permutation groups}

\author{Teodor Banica}
\address{T.B.: Department of Mathematics, Cergy-Pontoise University, 95000 Cergy-Pontoise, France. {\tt teodor.banica@u-cergy.fr}}

\author{Ion Nechita}
\address{I.N.: Department of Theoretical Physics, Paul Sabatier University, 31062 Toulouse, France. {\tt nechita@irsamc.ups-tlse.fr}}

\subjclass[2000]{16T05 (46L54)}
\keywords{Quantum permutation, Matrix model}

\begin{abstract}
We study the matrix models $\pi:C(S_N^+)\to M_N(C(X))$ which are flat, in the sense that the standard generators of $C(S_N^+)$ are mapped to rank 1 projections. Our first result is a generalization of the Pauli matrix construction at $N=4$, using finite groups and 2-cocycles. Our second result is the construction of a universal representation of $C(S_N^+)$, inspired from the Sinkhorn algorithm, that we conjecture to be inner faithful.
\end{abstract}

\maketitle

\section*{Introduction}

The quantum permutation group $S_N^+$ was introduced by Wang in \cite{wa1}. Of particular interest are the quantum subgroups $\mathcal G\subset S_N^+$ appearing from random matrix representations $\pi:C(S_N^+)\to M_N(C(X))$ via the Hopf image construction \cite{bb2}. One key problem is the computation of the law of the main character of $\mathcal G$. See \cite{bb3}, \cite{bco}, \cite{bfs}.

A number of general algebraic and analytic tools for dealing with such questions have been developed \cite{bb2}, \cite{bfs}, \cite{bic}, \cite{chi}, \cite{wa2}. However, at the level of concrete examples, only two types of models $\pi:C(S_N^+)\to M_N(C(X))$ have been succesfully investigated, so far. The first example, coming from the Pauli matrices, was investigated in \cite{bco}. The second example, coming from deformed Fourier matrices, was investigated in \cite{bb3}.

Our purpose here is to advance on such questions:
\begin{enumerate}
\item The Pauli matrix construction and the deformed Fourier matrix one are both of type $\pi:C(S_N^+)\to C(U_B,\mathcal L(B))$, with $B$ being a finite dimensional $C^*$-algebra. We will investigate here the case where $B=C^*_\sigma(G)$ is a cocycle twist of a finite group algebra, which generalizes the Pauli matrix construction. Our main result will be the computation of the law of the main character.

\item We will present as well a ``universal'' construction, inspired from the Sinkhorn algorithm \cite{sin}, \cite{skn}. This algorithm starts with a $N\times N$ matrix having positive entries and produces, via successive averagings over rows/columns, a bistochastic matrix. We will find here an adaptation of this algorithm to Wang's magic unitaries \cite{wa1}, which conjecturally produces an inner faithful representation of $C(S_N^+)$.   
\end{enumerate}

There are of course many questions raised by the present work. Regarding the generalized Pauli matrix construction, our results, and also \cite{bb1}, \cite{bbc}, suggest that the associated quantum group should be a twist of $PU_n$. Also, this construction still remains to be unified with the deformed Fourier matrix one. Regarding the Sinkhorn type models, here our computer simulations suggest that we should get a free Poisson law \cite{mpa}, \cite{vdn}, but so far, we have no convincing abstract methods in order to approach this question.

The paper is organized as follows: 1-2 contain preliminaries and generalities, in 3-4 we study the generalized Pauli models, and in 5-6 we study the Sinkhorn type models.

\medskip

\noindent {\bf Acknowledgements.} The present work was started at the Fields Institute conference ``Quantum groups and Quantum information theory'', Herstmonceux 2015, and we would like to thank the organizers for the invitation. IN received financial support from the ANR grants RMTQIT ANR-12-IS01-0001-01 and StoQ ANR-14-CE25-0003-01.

\section{Quantum permutations}

We are interested in what follows in the quantum permutation group $S_N^+$, and in the random matrix representations of the associated Hopf algebra $C(S_N^+)$. 

Our starting point is the following notion, coming from Wang's paper \cite{wa1}:

\begin{definition}
A magic unitary is a square matrix over a $C^*$-algebra, $u\in M_N(A)$,  whose entries are projections, summing up to $1$ on each row and each column. 
\end{definition}

At $N=2$ these matrices are as follows, with $p$ being a projection:
$$u=\begin{pmatrix}p&1-p\\ 1-p&p\end{pmatrix}$$

At $N=3$ it is known from \cite{wa1} that the entries of $u$ must commute as well. At $N\geq 4$ the entries of $u$ no longer automatically commute. Indeed, we have here the following example, with $p,q\in B(H)$ being non-commuting projections:
$$u=\begin{pmatrix}p&1-p&0&0\\ 1-p&p&0&0\\0&0&q&1-q\\0&0&1-q&q\end{pmatrix}$$

The following key definition is due to Wang \cite{wa1}:

\begin{definition}
$C(S_N^+)$ is the universal $C^*$-algebra generated by the entries of  a $N\times N$ magic unitary matrix $w=(w_{ij})$, with the morphisms defined by
$$\Delta(w_{ij})=\sum_kw_{ik}\otimes w_{kj}\quad,\quad\varepsilon(u_{ij})=\delta_{ij}\quad,\quad S(u_{ij})=u_{ji}$$
as comultiplication, counit and antipode. 
\end{definition}

This algebra satisfies Woronowicz' axioms in \cite{wo1}, \cite{wo2}, and the underlying space $S_N^+$ is therefore a compact quantum group, called quantum permutation group.

Observe that any magic unitary $u\in M_N(A)$ produces a representation $\pi:C(S_N^+)\to A$, given by $\pi(w_{ij})=u_{ij}$. In particular, we have a representation as follows:
$$\pi:C(S_N^+)\to C(S_N)\quad:\quad w_{ij}\to\chi\left(\sigma\in S_N\big|\sigma(j)=i\right)$$

The corresponding embedding $S_N\subset S_N^+$ is an isomorphism at $N=2,3$, but not at $N\geq4$, where $S_N^+$ is infinite. Moreover, it is known that we have $S_4^+\simeq SO_3^{-1}$, and that any $S_N^+$ with $N\geq4$ has the same fusion semiring as $SO_3$. See \cite{bb2}, \cite{bco}.

Our claim now is that, given a magic unitary $u\in M_N(A)$, we can associate to it a certain quantum permutation group $\mathcal G\subset S_N^+$. In order to perform this construction, we use the notions of Hopf image and inner faithfulness, from \cite{bb2}:

\begin{definition}
The Hopf image of a $C^*$-algebra representation $\pi:C(\mathcal G)\to A$ is the smallest Hopf $C^*$-algebra quotient $C(\mathcal G)\to C(\mathcal G')$ producing a factorization as follows:
$$\pi:C(\mathcal G)\to C(\mathcal G')\to A$$
The representation $\pi$ is called inner faithful when $\mathcal G=\mathcal G'$.
\end{definition}

Here $\mathcal G$ can be any compact quantum group, in the sense of \cite{wo1}, \cite{wo2}.

As a basic example, when $\mathcal G=\widehat{\Gamma}$ is a group dual, $\pi:C^*(\Gamma)\to A$ must come from a unitary group representation $\Gamma\to U_A$, and the minimal factorization is the one obtained by taking the image, $\Gamma\to\Gamma'\subset U_A$. Thus $\pi$ is inner faithful when $\Gamma\subset U_A$.

Also, given a compact group $\mathcal G$, and elements $g_1,\ldots,g_K\in \mathcal G$, we can consider the representation $\pi=\oplus_iev_{g_i}:C(\mathcal G)\to\mathbb C^K$. The minimal factorization of $\pi$ is then via $C(\mathcal G')$, with $\mathcal G'=\overline{<g_1,\ldots,g_K>}$. Thus $\pi$ is inner faithful when $\mathcal G=\overline{<g_1,\ldots,g_K>}$.

Now back to our above claim, we can now formulate:

\begin{definition}
Associated to any magic unitary $u\in M_N(A)$ is the smallest quantum permutation group $\mathcal G\subset S_N^+$ producing a factorization
$$\pi:C(S_N^+)\to C(\mathcal G)\to A$$
of the representation $\pi:C(S_N^+)\to A$ given by $w_{ij}\to u_{ij}$.
\end{definition}

At the level of examples, let us recall that a Latin square is a matrix $L\in M_N(1,\ldots,N)$ having the property that each of its rows and columns is a permutation of $1,\ldots,N$. For instance, associated to any finite group $\mathcal G$ is the Latin square $(L_\mathcal G)_{ij}=ij^{-1}$, with $i,j,ij^{-1}\in \mathcal G$ being regarded as elements of $\{1,\ldots,N\}$, where $N=|\mathcal G|$.

With these conventions, we have the following result:

\begin{theorem}
If $u\in M_N(A)$ comes from a Latin square $L\in M_N(1,\ldots,N)$, in the sense that $u_{ij}=p_{L_{ij}}$, with $p_1,\ldots,p_N\in A$ being projections summing up to $1$, then:
\begin{enumerate}
\item $\mathcal G\subset S_N^+$ is the subgroup of $S_N$ generated by the rows of $L$.

\item In particular, when $L=L_\mathcal G$, we obtain the group $\mathcal G$ itself.

\item In addition, this is the only case where $\mathcal G$ is classical.
\end{enumerate}
\end{theorem}

\begin{proof}
These results are well-known, the proof being as follows:

(1) This comes from the fact that we have a factorization $\pi:C(S_N^+)\to C(\mathcal G)\subset A$.

(2) This follows from (1), because the rows of $L_\mathcal G$ generate the group $\mathcal G$ itself.

(3) This follows by using the Gelfand theorem. For details here, see \cite{bb2}.
\end{proof}

\section{Cocyclic models}

We are interested in what follows in representations of type $\pi:C(S_N^+)\to M_N(C(X))$, and in the computation of their Hopf images. As a motivation, it is known that the existence of an inner faithful representation of type $\pi:C(\mathcal G)\to M_N(C(X))$ implies that $L^\infty(\mathcal G)$ has the Connes embedding property. For a discussion here, see \cite{bfs}, \cite{bcv}, \cite{chi}.

The key example of a magic unitary matrix $u\in M_N(A)$ over a random matrix algebra, $A=M_N(C(X))$, appears at $N=4$, in connection with the Pauli matrices:
$$g_1=\begin{pmatrix}1&0\\ 0&1\end{pmatrix}\qquad
g_2=\begin{pmatrix}i&0\\ 0&-i\end{pmatrix}\qquad
g_3=\begin{pmatrix}0&1\\-1&0\end{pmatrix}\qquad
g_4=\begin{pmatrix}0&i\\ i&0\end{pmatrix}$$

Given a vector $\xi$, we denote by $Proj(\xi)$ the rank 1 projection onto the space $\mathbb C\xi$. 

We have the following result, from \cite{bco}:

\begin{proposition}
We have a representation, as follows,
$$\pi:C(S_4^+)\to M_4(C(U_2))\quad:\quad w_{ij}\to[x\to Proj(g_ixg_j^*)]$$
which commutes with canonical integration maps, and is faithful.
\end{proposition}

\begin{proof}
Since the elements $g_ixg_j^*\in U_2\subset M_2(\mathbb C)\simeq\mathbb C^4$ are pairwise orthogonal, when $i$ is fixed and $j$ varies, or vice versa, the corresponding rank 1 projections form a magic unitary, and so we have a representation as in the statement. 

The point now is that the combinatorics of the variables $x\to Proj(g_ixg_j^*)$ can be shown to be the same as the Weingarten combinatorics of the variables $w_{ij}\in C(S_4^+)$. This gives the integration assertion, and the faithfulness assertion follows from it. See \cite{bco}.
\end{proof}

At $N\geq5$ now, since the dual of $S_N^+$ is not amenable, we cannot have a faithful representation $\pi:C(S_N^+)\to M_N(C(X))$. Our purpose will be find such a representation which is inner faithful, or at least which is ``as inner faithful'' as possible.

Assume that $B$ is a $C^*$-algebra, of finite dimension $\dim B=N<\infty$. We can endow $B$ with its canonical trace, $tr:B\subset\mathcal L(B)\to\mathbb C$, and use the scalar product $<a,b>=tr(ab^*)$. We recall that, in terms of the decomposition $B=\oplus_sM_{n_s}(\mathbb C)$, we have $N=\sum_sn_s^2$, and the weights of the canonical trace are  $tr(I_s)=n_s^2/N$.

With these conventions, we can formulate:

\begin{definition}
A magic unitary $u\in M_N(\mathcal L(B))$ is called:
\begin{enumerate}
\item Flat, if each $u_{ij}\in\mathcal L(B)$ is a rank $1$ projection.

\item Split, if $u_{ij}=Proj(e_if_j^*)$, for certain sets $\{e_i\},\{f_i\}\subset U_B$.

\item Fully split, if $u_{ij}=Proj(g_ixg_j^*)$, with $\{g_i\}\subset U_B$, and $x\in U_B$.
\end{enumerate}
\end{definition}

Observe that the above sets $\{e_i\},\{f_i\},\{g_i\}\subset U_B$ must consist of pairwise orthogonal unitaries. As an example, for $B=M_2(\mathbb C)$ we have $U_B=U_2$, and since $\{g_1,g_2,g_3,g_4\}\subset U_2$ is an orthogonal basis, the representation in Proposition 2.1 is fully split.

Let us first discuss the case $B=\mathbb C^N$. We recall that a complex Hadamard matrix is a square matrix $H\in M_N(\mathbb T)$, whose rows $H_1,\ldots,H_N\in\mathbb T^N$ are pairwise orthogonal. The basic example is the Fourier coupling $F_G(i,a)=<i,a>$ of a finite abelian group $G$, regarded as square matrix, $F_G\in M_{G,\widehat{G}}(\mathbb C)$. With these conventions, we have:

\begin{proposition}
The flat magic unitaries over $B=\mathbb C^N$ are as follows: 
\begin{enumerate}
\item The split ones are $u_{ij}=Proj(H_i/K_j)$, with $H,K\in M_N(\mathbb C)$ Hadamard.

\item The fully split ones are $u_{ij}=Proj(H_i/H_j)$, with $H\in M_N(\mathbb C)$ Hadamard.

\item If $G$ is an abelian group, $|G|=N$, then $u_{ij}=Proj((F_G)_{i-j})$ is fully split.
\end{enumerate}
\end{proposition}

\begin{proof}
For the algebra $B=\mathbb C^N$ the unitary group is $U_B=\mathbb T^N$, and the condition that that $g_1,\ldots,g_N\in U_B$ satisfy $<g_i,g_j>=\delta_{ij}$ is equivalent to the fact that the $N\times N$ matrix having $g_1,\ldots,g_N\in\mathbb T^N$ as row vectors is Hadamard. But this gives (1) and (2), and (3) is clear from (2), since the Fourier matrix $F_G$ is Hadamard.
\end{proof}

Let us clarify now the relation with Theorem 1.5. We first have:

\begin{proposition}
The split magic unitaries which produce Latin squares are those of the form $u_{ij}=Proj(g_ig_j^*)$, with $\{g_1,\ldots,g_N\}\subset U_B$ being pairwise orthogonal, and forming a group $G\subset PU_B$. For such a magic unitary, the associated Latin square is $L_G$.
\end{proposition}

\begin{proof}
Assume indeed that $u_{ij}=Proj(e_if_j^*)$ produces a Latin square.

(1) Our first claim is that we can assume $e_1=f_1=1$. Indeed, given $x,y\in U_B$ the matrix $u_{ij}'=Proj(xe_if_j^*y)$ is still magic, and in the case where $u$ comes from a Latin square, $u_{ij}=Proj(\xi_{L_{ij}})$, we have $u_{ij}'=Proj(\xi_{L_{ij}}')$ with $\xi'_{ab}=x\xi_{ab}y$, and so $u'$ comes from $L$ as well. Thus, by taking $x=e_1^*,y=f_1$, we can assume $e_1=f_1=1$.

(2) Our second claim is that we can assume $u_{ij}=Proj(e_ie_j^*)$. Indeed, since $u$ is magic, the first row of vectors $\{1,f_2^*,\ldots,f_N^*\}\subset PU_B$ must appear as a permutation of the first column of vectors $\{1,e_2,\ldots,e_N\}\subset PU_B$. Thus, up to a permutation of the columns, and a rescaling of the columns by elements in $Z(U_B)$, we can assume $f_i=e_i$, and we obtain $u_{ij}=Proj(e_ie_j^*)$. Observe that this permutation/rescaling of the columns won't change the fact that the associated Latin square $L$ comes or not from a group.

(3) Let us construct now $G$. The Latin square condition shows that for any $i,j$ there is a unique $k$ such that $e_ie_j=e_k$ inside $PU_B$, and our claim is that the operation $(i,j)\to k$ gives a group structure on the set of indices. Indeed, all the group axioms are clear from definitions, and we obtain in this way a subgroup $G\subset PU_B$, having order $N$. 

(4) With $G$ being constructed as above, we have $u_{ij}=Proj(e_ie_j^*)=Proj(e_{ij^{-1}})$. Thus we have $u_{ij}'=Proj(\xi_{L_{ij}})$ with $\xi_k=e_k$ and $L_{ij}=ij^{-1}$, and we are done.
\end{proof}

In order to further process the above result, we will need:

\begin{definition}
A $2$-cocycle on a group $G$ is a function $\sigma:G\times G\to\mathbb T$ satisfying:
$$\sigma(gh,k)\sigma(g,h)=\sigma(g,hk)\sigma(h,k)$$
$$\sigma(g,1)=\sigma(1,g)=1$$
The algebra $C^*(G)$, with multiplication $g\cdot h=\sigma(g,h)gh$, is denoted $C^*_\sigma(G)$.
\end{definition}

Observe that $g\cdot h=\sigma(g,h)gh$ is associative, and that we have $g\cdot 1=1\cdot g=g$, due to the 2-cocycle condition. Thus $C^*_\sigma(G)$ is an associative algebra with unit 1. In fact, $C^*_\sigma(G)$ is a $C^*$-algebra, with the involution making the canonical generators $g\in C^*_\sigma(G)$ unitaries. The canonical trace on $C^*_\sigma(G)$ coincides then with that of $C^*(G)$.

With this notion in hand, we can now formulate:

\begin{proposition}
The split magic unitaries which produce Latin squares are precisely those of the form $u_{ij}=Proj(g_ig_j^*)$, with $\{g_1,\ldots,g_N\}$ being the standard basis of a twisted group algebra $C^*_\sigma(G)$. In this case, the associated Latin square is $L_G$.
\end{proposition}

\begin{proof}
We use Proposition 2.4. With the notations there, $\{g_1,\ldots,g_N\}\subset U_B$ must form a group $G\subset PU_B$, and so there are scalars $\sigma(i,j)\in\mathbb T$ such that $g_ig_j=\sigma(i,j)g_{ij}$.

It follows from definitions that $\sigma$ is a 2-cocycle, and our claim now is that we have $B=C^*_\sigma(G)$. Indeed, this is clear when $\sigma=1$, because by linear independence we can define a linear space isomorphism $B\simeq C^*(G)$, which follows to be a $C^*$-algebra isomorphism. In the general case, where $\sigma$ is arbitrary, the proof is similar.
\end{proof}

At the level of examples now, we can use the following construction:

\begin{proposition}
Let $H$ be a finite abelian group. 
\begin{enumerate}
\item The map $\sigma((i,a),(j,b))=<i,b>$ is a $2$-cocycle on $G=H\times\widehat{H}$.

\item We have an isomorphism of algebras $C^*_\sigma(G)\simeq M_n(\mathbb C)$, where $n=|H|$.

\item For $H=\mathbb Z_2$, the standard basis of $C^*_\sigma(G)$ is formed by multiples of $g_1,\ldots,g_4$.
\end{enumerate}
\end{proposition}

\begin{proof}
These results are all well-known, the proof being as follows:

(1) The map $\sigma:G\times G\to\mathbb T$ is a bicharacter, and is therefore a 2-cocycle. 

(2) Consider the Hilbert space $l^2(H)\simeq\mathbb C^n$, and let $\{E_{ij}|i,j\in H\}$ be the standard basis of $\mathcal L(l^2(H))\simeq M_n(\mathbb C)$. We define a linear map, as follows:
$$\varphi:C^*_\sigma(G)\to M_n(\mathbb C)\quad,\quad
g_{ia}\to\sum_k<k,a>E_{k,k+i}$$

The fact that $\varphi$ is multiplicative follows from:
\begin{eqnarray*}
\varphi(g_{ia})\varphi(g_{jb})
&=&\sum_k<k,a>E_{k,k+i}\sum_{k'}<k'+i,b>E_{k'+i,k'+i+j}\\
&=&\sum_k<k,a+b><i,b>E_{k,k+i+j}\\
&=&<i,b>\varphi(g_{i+j,a+b})
=\varphi(g_{ia}g_{jb})
\end{eqnarray*}

Recall now that the involution of $C^*_\sigma(G)$ is the one making the canonical generators $g\in C^*_\sigma(G)$ unitaries. Since we have $g_{ia}g_{-i,-a}=<-i,a>g_{00}=<-i,a>$, it follows that we have $g_{ia}^*=<i,a>g_{-i,-a}$, and the involutivity check goes as follows:
$$\varphi(g_{ia})^*=\sum_k<k,-a>E_{k+i,k}=\sum_l<l-i,-a>E_{l,l-i}=\varphi(g_{ia}^*)$$

In order to prove the bijectivity of $\varphi$, consider the following linear map:
$$\psi:M_n(\mathbb C)\to C^*_\sigma(G)\quad,\quad
E_{ij}\to\frac{1}{n}\sum_b<i,b>g_{j-i,-b}$$

It is routine to check that $\varphi,\psi$ are inverse to each other, and this finishes the proof.

(3) Consider first an arbitrary cyclic group $H=\mathbb Z_n$, written additively. We have then an identification $\widehat{\mathbb Z}_n\simeq\mathbb Z_n$, with the coupling being $<i,a>=w^{ia}$, where  $w=e^{2\pi i/n}$. Thus the above cocycle, written $\sigma:\mathbb Z_n^2\times\mathbb Z_n^2\to\mathbb T$, is given by $\sigma((i,a),(j,b))=w^{ib}$, and we have an isomorphism $\varphi:C^*_\sigma(\mathbb Z_n^2)\simeq M_n(\mathbb C)$, the formula being:
$$\varphi(g_{ia})=\sum_kw^{ka}E_{k,k+i}$$

At $n=2$ now, the root of unity is $w=-1$, we have $\varphi(g_{ia})=\sum_k(-1)^{ka}E_{k,k+i}$, and $\varphi:C^*_\sigma(\mathbb Z_2^2)\simeq M_2(\mathbb C)$ maps therefore $g_{00},g_{01},g_{11},g_{10}$ to the following matrices:
$$g_1'=\begin{pmatrix}1&0\cr 0&1\end{pmatrix}\qquad
g_2'=\begin{pmatrix}1&0\cr 0&-1\end{pmatrix}\qquad
g_3'=\begin{pmatrix}0&1\cr -1&0\end{pmatrix}\qquad
g_4'=\begin{pmatrix}0&1\cr 1&0\end{pmatrix}$$

But these matrices are proportional, by factors $1,i,1,i$, to the Pauli matrices.
\end{proof}

We have now all the needed ingredients for generalizing the Pauli matrix construction. The result here, conceptually motivated by Proposition 2.6 above, is as follows:

\begin{theorem}
Given a $2$-cocycle $\sigma:G\times G\to\mathbb T$, we have a representation
$$\pi:C(S_N^+)\to C(U_B,\mathcal L(B))\quad:\quad w_{ij}\to[x\to Proj(g_ixg_j^*)]$$
where $\{g_1,\ldots,g_N\}\subset U_B$ is the standard basis of the algebra $B=C^*_\sigma(G)$. Moreover:
\begin{enumerate}
\item As an example, we can use $G=H\times\widehat{H}$, with $\sigma((i,a),(j,b))=<i,b>$.

\item For $G=\mathbb Z_2\times\mathbb Z_2$ with such a cocycle, we obtain the Pauli representation.

\item When the cocycle is trivial, we obtain the Fourier matrix representation.
\end{enumerate}
\end{theorem}

\begin{proof}
The first assertion follows from Proposition 2.6 and its proof, (1) and (2) follow from Proposition 2.7, and (3) follows from Proposition 2.3.
\end{proof}

We should mention that the ``deformed Fourier'' representations in \cite{bb3} are as well of the form $\pi:C(S_N^+)\to C(U_B,\mathcal L(B))$, with $B=\mathbb C^{mn}$. Unifying these representations with those constructed above is an open question, that we would like to raise here. 

\section{Laws of characters}

In order to study the matrix model representations of type $\pi:C(S_N^+)\to M_K(C(X))$, we can use functional analytic technology from \cite{bfs}, \cite{wa2}. Assume indeed that $X$ is a compact probability space, so that the target algebra has a trace $tr:M_K(C(X))\to\mathbb C$, given by $tr(M)=\frac{1}{K}\sum_{i=1}^K\int_XM_{ii}(x)dx$. We have then the following result:

\begin{proposition}
Let $\pi:C(S_N^+)\to C(\mathcal G)\to M_K(C(X))$ be a Hopf image factorization, mapping $w_{ij}\to v_{ij}\to u_{ij}$, and let $\chi=\sum_iv_{ii}$.
\begin{enumerate}
\item $\int_\mathcal G=\lim_{k\to\infty}\frac{1}{k}\sum_{r=1}^k\int_\mathcal G^r$, with $\int_\mathcal G^r=(tr\circ\pi)^{*r}$, where $\phi*\psi=(\phi\otimes\psi)\Delta$.

\item $\int_\mathcal G^rv_{i_1j_1}\ldots v_{i_pj_p}=(T_p^r)_{i_1\ldots i_p,j_1\ldots j_p}$, where $(T_p)_{i_1\ldots i_p,j_1\ldots j_p}=tr(u_{i_1j_1}\ldots u_{i_pj_p})$.

\item The moments of $\chi$ with respect to $\int_\mathcal G^r$ are the numbers $c_p^r=Tr(T_p^r)$.
\end{enumerate}
\end{proposition}

\begin{proof}
The first assertion, which is the key one, was proved in \cite{bfs} in the case $X=\{.\}$, and then in \cite{wa2} in the general, parametric case. The second assertion is elementary, and the third one follows from it, by summing over indices $i_k=j_k$.
\end{proof}

As a main consequence, if we denote by $\mu,\mu^r$ the laws of the main character $\chi$ with respect to the Haar functional $\int_\mathcal G$, and with its truncated version $\int_\mathcal G^r=(tr\circ\pi)^{*r}$, we have a convergence in moments $\mu^r\to\mu$. Following now \cite{bb3}, we have:

\begin{proposition}
For a representation coming from a split matrix, $u_{ij}=Proj(e_if_j^*)$, the truncated measure $\mu^r$ is the law of the Gram matrix of the vectors
$$\xi_{i_1\ldots i_r}=e_{i_1}f_{i_2}^*\otimes e_{i_2}f_{i_3}^*\otimes\ldots\ldots\otimes e_{i_r}f_{i_1}^*$$
with respect to the normalized trace of the $N^r\times N^r$ matrices.
\end{proposition}

\begin{proof}
According to Proposition 3.1 (3), the moments of $\mu^r$ are given by:
\begin{eqnarray*}
c_p^r
&=&\sum_{i_1^1\ldots i_p^r}(T_p)_{i_1^1\ldots i_p^1,i_1^2\ldots i_p^2}(T_p)_{i_1^2\ldots i_p^2,i_1^3\ldots i_p^3}\ldots\ldots(T_p)_{i_1^r\ldots i_p^r,i_1^1\ldots i_p^1}\\
&=&\sum_{i_1^1\ldots i_p^r}tr(u_{i_1^1i_1^2}\ldots u_{i_p^1i_p^2})tr(u_{i_1^2i_1^3}\ldots u_{i_p^2i_p^3})\ldots\ldots tr(u_{i_1^ri_1^1}\ldots u_{i_p^ri_p^1})
\end{eqnarray*}

In the case of a split magic unitary, $u_{ij}=Proj(e_if_j^*)$, since the vectors $e_if_j^*$ are all of norm 1, with respect to the canonical scalar product, we therefore obtain:
\begin{eqnarray*}
c_p^r
&=&\frac{1}{N^r}\sum_{i_1^1\ldots i_p^r}<e_{i_1^1}f_{i_1^2}^*,e_{i_2^1}f_{i_2^2}^*>\ldots<e_{i_p^1}f_{i_p^2}^*,e_{i_1^1}f_{i_1^2}^*>\\
&&\hskip15.6mm<e_{i_1^2}f_{i_1^3}^*,e_{i_2^2}f_{i_2^3}^*>\ldots<e_{i_p^2}f_{i_p^3}^*,e_{i_1^2}f_{i_1^3}^*>\\
&&\hskip43mm\ldots\ldots\\
&&\hskip15.6mm<e_{i_1^r}f_{i_1^1}^*,e_{i_2^r}f_{i_2^1}^*>\ldots<e_{i_p^r}f_{i_p^1}^*,e_{i_1^r}f_{i_1^1}^*>
\end{eqnarray*}

Now by changing the order of the terms in the product, this gives:
\begin{eqnarray*}
c_p^r
&=&\frac{1}{N^r}\sum_{i_1^1\ldots i_p^r}<e_{i_1^1}f_{i_1^2}^*,e_{i_2^1}f_{i_2^2}^*><e_{i_1^2}f_{i_1^3}^*,e_{i_2^2}f_{i_2^3}^*>\ldots<e_{i_1^r}f_{i_1^1}^*,e_{i_2^r}f_{i_2^1}^*>\\
&&\hskip43mm\ldots\ldots\\
&&\hskip15.6mm<e_{i_p^1}f_{i_p^2}^*,e_{i_1^1}f_{i_1^2}^*><e_{i_p^2}f_{i_p^3}^*,e_{i_1^2}f_{i_1^3}^*>\ldots<e_{i_p^r}f_{i_p^1}^*,e_{i_1^r}f_{i_1^1}^*>
\end{eqnarray*}

In terms of the vectors $\xi_{i_1\ldots i_r}=e_{i_1}f_{i_2}^*\otimes\ldots\otimes e_{i_r}f_{i_1}^*$ in the statement, and then of their Gram matrix $G_{i_1\ldots i_r,j_1\ldots j_r}=<\xi_{i_1\ldots i_r},\xi_{j_1\ldots j_r}>$, we obtain the following formula:
\begin{eqnarray*}
c_p^r&=&\frac{1}{N^r}\sum_{i_1^1\ldots i_p^r}<\xi_{i_1^1\ldots i_1^r},\xi_{i_2^1\ldots i_2^r}>\ldots\ldots<\xi_{i_p^1\ldots i_p^r},\xi_{i_1^1\ldots i_1^r}>\\
&=&\frac{1}{N^r}\sum_{i_1^1\ldots i_p^r}G_{i_1^1\ldots i_1^r,i_2^1\ldots i_2^r}\ldots\ldots G_{i_p^1\ldots i_p^r,i_1^1\ldots i_1^r}\\
&=&\frac{1}{N^r}Tr(G^p)=tr(G^p)
\end{eqnarray*}

But this gives the formula in the statement, and we are done.
\end{proof}

In the fully split case now, we have the following result:

\begin{theorem}
For a representation coming from a fully split matrix, $u_{ij}=Proj(g_ixg_j^*)$, the truncated measure $\mu^r$ is the law of the Gram matrix of the vectors
$$\xi_{i_1\ldots i_r}^{x_1\ldots x_r}=g_{i_1}x_1g_{i_2}^*\otimes g_{i_2}x_2g_{i_3}^*\otimes\ldots\ldots\otimes g_{i_r}x_rg_{i_1}^*$$
with respect to the usual integration over $M_{N^r}(C(U_B^r))$.
\end{theorem}

\begin{proof}
The idea is that the computations in the proof of Proposition 3.2 apply, with $e_i=g_ix$ and $f_i=g_i$, and with an integral $\int_{U_B^r}$ added. To be more precise, we can start with the same formula as there, stating that the moments of $\mu^r$ are given by:
$$c_p^r=\sum_{i_1^1\ldots i_p^r}tr(u_{i_1^1i_1^2}\ldots u_{i_p^1i_p^2})\ldots\ldots tr(u_{i_1^ri_1^1}\ldots u_{i_p^ri_p^1})$$

In the case of a fully split matrix, $u_{ij}=Proj(g_ixg_j^*)$, since the vectors $g_ixg_j^*$ are all of norm 1, we therefore obtain:
\begin{eqnarray*}
c_p^r
&=&\frac{1}{N^r}\sum_{i_1^1\ldots i_p^r}\int_{U_B}<g_{i_1^1}x_1g_{i_1^2}^*,g_{i_2^1}x_1g_{i_2^2}^*>\ldots<g_{i_p^1}x_1g_{i_p^2}^*,g_{i_1^1}x_1g_{i_1^2}^*>dx_1\\
&&\hskip58mm\ldots\ldots\\
&&\hskip15.6mm\int_{U_B}<g_{i_1^r}x_rg_{i_1^1}^*,g_{i_2^r}x_rg_{i_2^1}^*>\ldots<g_{i_p^r}x_rg_{i_p^1}^*,g_{i_1^r}x_rg_{i_1^1}^*>dx_r
\end{eqnarray*}

Now by changing the order of the terms in the product, this gives:
\begin{eqnarray*}
c_p^r
&=&\frac{1}{N^r}\sum_{i_1^1\ldots i_p^r}\int_{U_B^r}<g_{i_1^1}x_1g_{i_1^2}^*,g_{i_2^1}x_1g_{i_2^2}^*>\ldots<g_{i_1^r}x_rg_{i_1^1}^*,g_{i_2^r}x_rg_{i_2^1}^*>\\
&&\hskip58mm\ldots\ldots\\
&&\hskip23mm<g_{i_p^1}x_1g_{i_p^2}^*,g_{i_1^1}x_1g_{i_1^2}^*>\ldots<g_{i_p^r}x_rg_{i_p^1}^*,g_{i_1^r}x_rg_{i_1^1}^*>dx
\end{eqnarray*}

In terms of the vectors $\xi_{i_1\ldots i_r}^{x_1\ldots x_r}=g_{i_1}x_1g_{i_2}^*\otimes\ldots\ldots\otimes g_{i_r}x_rg_{i_1}^*$  in the statement, and then of their Gram matrix $G_{i_1\ldots i_r,j_1\ldots j_r}^{x_1\ldots x_r}=<\xi_{i_1\ldots i_r}^{x_1\ldots x_r},\xi_{j_1\ldots j_r}^{x_1\ldots x_r}>$, we therefore obtain:
\begin{eqnarray*}
c_p^r&=&\frac{1}{N^r}\int_{U_B^r}\sum_{i_1^1\ldots i_p^r}<\xi_{i_1^1\ldots i_1^r}^{x_1\ldots x_r},\xi_{i_2^1\ldots i_2^r}^{x_1\ldots x_r}>\ldots\ldots<\xi_{i_p^1\ldots i_p^r}^{x_1\ldots x_r},\xi_{i_1^1\ldots i_1^r}^{x_1\ldots x_r}>dx\\
&=&\frac{1}{N^r}\int_{U_B^r}\sum_{i_1^1\ldots i_p^r}G_{i_1^1\ldots i_1^r,i_2^1\ldots i_2^r}^{x_1\ldots x_r}\ldots\ldots G_{i_p^1\ldots i_p^r,i_1^1\ldots i_1^r}^{x_1\ldots x_r}dx\\
&=&\frac{1}{N^r}\int_{U_B^r}Tr((G^{x_1\ldots x_r})^p)dx\\
&=&\int_{U_B^r}tr((G^{x_1\ldots x_r})^p)dx
\end{eqnarray*}

But this gives the formula in the statement, and we are done.
\end{proof}

\section{Cocyclic abelian models}

Let us go back now to the cocyclic abelian models, from Theorem 2.8 (1) above. We will explicitely compute the law of the main character, for these models.

By appplying the general formula in Theorem 3.3, we first have:

\begin{proposition}
For the representation $\pi:C(S_{n^2}^+)\to C(U_n,M_{n^2}(\mathbb C))$ coming from an abelian group $H$, with $|H|=n$, the truncated measure $\mu^r$ is the law of the matrix
\begin{eqnarray*}
G_{i_1a_1\ldots i_ra_r,j_1b_1\ldots j_rb_r}^{x_1\ldots x_r}
&=&\frac{1}{n^r}\sum_{p_1\ldots p_r}\sum_{s_1\ldots s_r}<p_1-s_r,a_1-b_1>\ldots<p_r-s_{r-1},a_r-b_r>\\
&&(x_1)_{p_1+i_1,s_1+i_2}\ldots(x_r)_{p_r+i_r,s_r+i_1}\cdot 
(\bar{x}_1)_{p_1+j_1,s_1+j_2}\ldots(\bar{x}_r)_{p_r+j_r,s_r+j_1}
\end{eqnarray*}
with respect to the usual integration over $M_{n^{2r}}(C(U_n^r))$.
\end{proposition}

\begin{proof}
We use the general formula found in Theorem 3.3 above. The Gram matrix that we are interested in, having now double indices, is given by:
\begin{eqnarray*}
G_{i_1a_1\ldots i_ra_r,j_1b_1\ldots j_rb_r}^{x_1\ldots x_r}
&=&<\xi_{i_1a_1\ldots i_ra_r}^{x_1\ldots x_r},\xi_{j_1b_1\ldots j_rb_r}^{x_1\ldots x_r}>\\
&=&<g_{i_1a_1}x_1g_{i_2a_2}^*,g_{j_1b_1}x_1g_{j_2b_2}^*>\ldots <g_{i_ra_r}x_rg_{i_1a_1}^*,g_{j_rb_r}x_rg_{j_1b_1}^*>
\end{eqnarray*}

In the case of a cocyclic abelian model, as in the statement, we can use for computations the isomorphism found in the proof of Proposition 2.7, namely:
$$C^*_\sigma(G)\simeq M_n(\mathbb C)\quad:\quad g_{ia}\to\sum_k<k,a>E_{k,k+i}$$

With this identification made, the scalar products can be computed as follows:
\begin{eqnarray*}
&&<g_{ia}xg_{jb}^*,g_{kc}xg_{ld}^*>\\
&=&tr(g_{ia}xg_{jb}^*g_{ld}x^*g_{kc}^*)\\
&=&\frac{1}{n}\sum_{pqrstu}(g_{ia})_{pq}x_{qr}(g_{jb}^*)_{rs}(g_{ld})_{st}(x^*)_{tu}(c_{kc}^*)_{up}\\
&=&\frac{1}{n}\sum_{pqrstu}\delta_{p+i,q}<p,a>x_{qr}\delta_{s+j,r}\overline{<s,b>}\delta_{s+l,t}<s,d>\bar{x}_{ut}\delta_{p+k,u}\overline{<p,c>}\\
&=&\frac{1}{n}\sum_{ps}<p,a-c><s,d-b>x_{p+i,s+j}\bar{x}_{p+k,s+l}
\end{eqnarray*}

Thus the Gram matrix that we are interested in is given by:
\begin{eqnarray*}
G_{i_1a_1\ldots i_ra_r,j_1b_1\ldots j_rb_r}^{x_1\ldots x_r}
&=&\frac{1}{n}\sum_{p_1s_1}<p_1,a_1-b_1><s_1,b_2-a_2>(x_1)_{p_1+i_1,s_1+i_2}(\bar{x}_1)_{p_1+j_1,s_1+j_2}\\
&&\hskip22mm\ldots\ldots\\
&&\frac{1}{n}\sum_{p_rs_r}<p_r,a_r-b_r><s_r,b_1-a_1>(x_r)_{p_r+i_r,s_r+i_1}(\bar{x}_r)_{p_r+j_r,s_r+j_1}
\end{eqnarray*}

But this gives the formula in the statement, and we are done.
\end{proof}

The point now is that the Gram matrix in Proposition 4.1 is circulant, and so is diagonal in Fourier transform. By diagonalizing it, we obtain the following result:

\begin{proposition}
For the representation $C(S_{n^2}^+)\to C(U_n,M_{n^2}(\mathbb C))$ as above, the measure $\mu^r$ is the law of the diagonal random matrix
$$\Lambda^{x_1\ldots x_r}_{k_1c_1\ldots k_rc_r}=\Big|Tr(W_{k_1c_1}x_1\ldots W_{k_rc_r}x_r)\Big|^2$$
over $U_n^r$, where $W_{kc}:e_i\to<k,i>e_{i+c}$ are the standard unitaries of $C^*_\sigma(\mathbb Z_n^2)\simeq M_n(\mathbb C)$.
\end{proposition}

\begin{proof}
As already mentioned, the idea will be that of applying a discrete Fourier transform. With $F_{ij}=\frac{1}{\sqrt{n}}<i,j>$, having as inverse $\bar{F}_{ij}=\frac{1}{\sqrt{n}}<-i,j>$, we have:
\begin{eqnarray*}
(F^{\otimes 2r}G^x\bar{F}^{\otimes 2r})_{kc,ld}
&=&\sum_{ijab}(F^{\otimes 2r})_{kc,ij}G^x_{ia,jb}(\bar{F}^{\otimes 2r})_{jb,ld}\\
&=&\frac{1}{n^{3r}}\sum_{ijabps}<k_1,i_1>\ldots<k_r,i_r><c_1,a_1>\ldots<c_r,a_r>\\
&&<-j_1,l_1>\ldots<-j_r,l_r><-b_1,d_1>\ldots<-b_r,d_r>\\
&&<p_1-s_r,a_1-b_1>\ldots<p_r-s_{r-1},a_r-b_r>\\
&&(x_1)_{p_1+i_1,s_1+i_2}\ldots(x_r)_{p_r+i_r,s_r+i_1} 
\cdot(\bar{x}_1)_{p_1+j_1,s_1+j_2}\ldots(\bar{x}_r)_{p_r+j_r,s_r+j_1}
\end{eqnarray*}

We can rewrite this formula in the following way:
\begin{eqnarray*}
(F^{\otimes 2r}G^x\bar{F}^{\otimes 2r})_{kc,ld}
&=&\frac{1}{n^r}\sum_{ijps}<k_1,i_1>\ldots<k_r,i_r><-j_1,l_1>\ldots<-j_r,l_r>\\
&&(x_1)_{p_1+i_1,s_1+i_2}\ldots(x_r)_{p_r+i_r,s_r+i_1}\cdot 
(\bar{x}_1)_{p_1+j_1,s_1+j_2}\ldots(\bar{x}_r)_{p_r+j_r,s_r+j_1}\\
&&\frac{1}{n}\sum_{a_1}<c_1+p_1-s_r,a_1>\ldots\frac{1}{n}\sum_{a_r}<c_r+p_r-s_{r-1},a_r>\\
&&\frac{1}{n}\sum_{b_1}<d_1+p_1-s_r,-b_1>\ldots\frac{1}{n}\sum_{b_r}<d_r+p_r-s_{r-1},-b_r>
\end{eqnarray*}

By summing over $a_i,b_i$, we must have $c_i=d_i$ and $s_{i-1}=c_i+p_i$. By changing the indices of summation, $i_x\to i_x-p_x$ and $j_x\to j_x-p_x$, we obtain:
\begin{eqnarray*}
(F^{\otimes 2r}G^x\bar{F}^{\otimes 2r})_{kc,ld}
&=&\frac{1}{n^r}\delta_{cd}\sum_{ijp}<k_1,i_1-p_1>\ldots<k_r,i_r-p_r>\\
&&<p_1-j_1,l_1>\ldots<p_r-j_r,l_r>\\
&&(x_1)_{i_1,i_2+c_2}\ldots(x_r)_{i_r,i_1+c_1}\cdot 
(\bar{x}_1)_{j_1,j_2+c_2}\ldots(\bar{x}_r)_{j_r,j_1+c_1}\\
&=&\delta_{cd}\sum_{ij}<k_1,i_1>\ldots<k_r,i_r><-j_1,l_1>\ldots<-j_r,l_r>\\
&&(x_1)_{i_1,i_2+c_2}\ldots(x_r)_{i_r,i_1+c_1}\cdot(\bar{x}_1)_{j_1,j_2+c_2}\ldots(\bar{x}_r)_{j_r,j_1+c_1}\\
&&\frac{1}{n}\sum_{p_1}<p_1,l_1-k_1>\ldots\frac{1}{n}\sum_{p_r}<p_r,l_r-k_r>\\
&=&\delta_{kl}\delta_{cd}\sum_{ij}<k_1,i_1-j_1>\ldots<k_r,i_r-j_r>\\
&&(x_1)_{i_1,i_2+c_2}\ldots(x_r)_{i_r,i_1+c_1}\cdot(\bar{x}_1)_{j_1,j_2+c_2}\ldots(\bar{x}_r)_{j_r,j_1+c_1}
\end{eqnarray*}

We conclude that $\mu^r$ is the law of the following diagonal random matrix:
$$\Lambda^{x_1\ldots x_r}_{k_1c_1\ldots k_rc_r}=\Big|\sum_i<k_1,i_1>\ldots<k_r,i_r>(x_1)_{i_1,i_2+c_2}\ldots(x_r)_{i_r,i_1+c_1}\Big|^2$$

Now observe that we have $<k,i>x_{i,j+c}=(A_kxB_c)_{ij}$, where $A_k:e_i\to<k,i>e_i$ and $B_c:e_i\to e_{i+c}$. In addition, we have $B_cA_k=W_{kc}$, and this gives:
\begin{eqnarray*}
\Lambda^{x_1\ldots x_r}_{k_1c_1\ldots k_rc_r}
&=&\Big|\sum_i(A_{k_1}x_1B_{c_2})_{i_1i_2}\ldots(A_{k_r}x_rB_{c_1})_{i_ri_1}\Big|^2
=\Big|Tr(A_{k_1}x_1B_{c_2}\ldots A_{k_r}x_rB_{c_1})\Big|^2\\
&=&\Big|Tr(B_{c_1}A_{k_1}x_1B_{c_2}\ldots A_{k_r}x_r)\Big|^2
=\Big|Tr(W_{k_1c_1}x_1\ldots W_{k_rc_r}x_r)\Big|^2
\end{eqnarray*}

Thus, we have obtained the formula in the statement.
\end{proof}

By making now some final manipulations, of probabilistic nature, everything simplifies in the formula in Proposition 4.2, and we obtain the following result:

\begin{theorem}
For a representation $\pi:C(S_{n^2}^+)\to C(U_n,M_{n^2}(\mathbb C))$ coming from an abelian group $H$, with $|H|=n$, all the measures $\mu^r$ are the laws of the following variable:
$$(x\in U_n)\to\Big|Tr(x)\Big|^2$$
In particular, $\mu$ coincides with the law of the main character of $PU_n=U_n/\mathbb T$.
\end{theorem}

\begin{proof}
We use the formula in Proposition 4.2 above. Observe first that the matrices $W_{kc}:e_i\to<k,i>e_{i+c}$ appearing there, called Weyl matrices, satisfy: 
\begin{eqnarray*}
W_{ia}^*&=&<i,a>W_{-i,-a}\\
W_{ia}W_{jb}&=&<i,b>W_{i+j,a+b}\\
W_{ia}W_{jb}^*&=&<j-i,b>W_{i-j,a-b}
\end{eqnarray*}

This is indeed already known from the cocyclic picture, and can be checked as well directly. Consider now the following group, obtained by tensoring such matrices:
$$W=\left\{W_{k_1c_1}\otimes\ldots\otimes W_{k_rc_r}\Big|k_i,c_i\in H\right\}$$

With these notions in hand, Proposition 4.2 tells us that $\mu^r$ appears as average over the above Weyl group $W$ of the laws of the following variables: 
$$(x\in U_n^r)\to\Big|Tr(W_{k_1c_1}x_1\ldots W_{k_rc_r}x_r)\Big|^2$$

The point now is that the random Weyl matrices $W_{k_ic_i}$ can be ``absorbed'' into the Haar distributed unitaries $x_i$, and we obtain that $\mu^r$ is the law of the following variable:
$$(x\in U_n^r)\to\Big|Tr(x_1\ldots x_r)\Big|^2$$

Now since the product $x_1\ldots x_r\in U_n$ is Haar distributed when the individual variables $x_1\in U_n,\ldots,x_r\in U_n$ are each Haar distributed, this gives the result.

Finally, the last assertion is clear, because $x\to Tr(x)$ is the character of the fundamental representation $\pi:U_n\to M_n(\mathbb C)$, and so $x\to|Tr(x)|^2$ is the character of $ad(\pi)$.
\end{proof}

Summarizing, we have obtained Diaconis-Shahshahani variables \cite{dsh}. The asymptotics can be investigated by using the Weingarten formula, and are well-known, see \cite{csn}, \cite{wei}. Note also that by \cite{rai}, the moments of the variable $|Tr(x)|^2$ are:
$$c_p=\#\left\{\sigma \in S_p\ \Big|\ \sigma \text{ has no increasing subsequence of length greater than } n\right\}$$

From a quantum group viewpoint, Theorem 4.3 suggests that the underlying quantum group should be a twist of $PU_n$. There is actually more evidence pointing towards this, coming from \cite{bb1}, \cite{bbc}. We intend to investigate these facts in some future work.

\section{Universal models}

We discuss in the reminder of this paper a ``universal'' model for $C(S_N^+)$. Generally speaking, the universal $K\times K$ model is simply the map $\pi_{univ}:C(S_N^+)\to M_N(C(Z_{N,K}))$ given by $\pi_{univ}(w_{ij})=(u\to u_{ij})$, where $Z_{N,K}$ is the space of all magic unitaries $u\in M_N(M_K(\mathbb C))$. However, not much is known about this space $Z_{N,K}$.

Our idea here is that of restricting attention to the case where $N=K$, and where $u\in M_N(M_N(\mathbb C))$ is ``flat'', in the sense that each $u_{ij}\in M_N(\mathbb C)$ is a rank 1 projection. Our main objective will be that of constructing an integration on the model space.

Given a flat magic unitary, we can write it, in a non-unique way, as $u_{ij}=Proj(\xi_{ij})$. The array $\xi=(\xi_{ij})$ is then a ``magic basis'', in the sense that each of its rows and columns is an orthonormal basis of $\mathbb C^N$. We are therefore led to two spaces, as follows:

\begin{definition}
Associated to any $N\in\mathbb N$ are the following spaces:
\begin{enumerate}
\item $X_N$, the space of all $N\times N$ flat magic unitaries $u=(u_{ij})$.

\item $K_N$, the space of all $N\times N$ magic bases $\xi=(\xi_{ij})$.
\end{enumerate}
\end{definition}

Let us recall now that the rank 1 projections $p\in M_N(\mathbb C)$ can be identified with the corresponding 1-dimensional subspaces $E\subset\mathbb C^N$, which are by definition the elements of the complex projective space $P^{N-1}_\mathbb C$. In addition, if we consider the complex sphere, $S^{N-1}_\mathbb C=\{z\in\mathbb C^N|\sum_i|z_i|^2=1\}$, we have a quotient map $\pi:S^{N-1}_\mathbb C\to P^{N-1}_\mathbb C$ given by $z\to Proj(z)$. Observe that $\pi(z)=\pi(z')$ precisely when $z'=wz$, for some $w\in\mathbb T$. 

Consider as well the embedding $U_N\subset(S^{N-1}_\mathbb C)^N$ given by $x\to(x_1,\ldots,x_N)$, where $x_1,\ldots,x_N$ are the rows of $x$. Finally, let us call an abstract matrix stochastic/bistochastic when the entries on each row/each row and column sum up to 1. 

With these notations, the abstract model spaces $X_N,K_N$ that we are interested in, and some related spaces, are as follows:

\begin{proposition}
We have inclusions and surjections as follows,
$$\begin{matrix}
K_N&\subset&U_N^N&\subset&M_N(S^{N-1}_\mathbb C)\\
\\
\downarrow&&\downarrow&&\downarrow\\
\\
X_N&\subset&Y_N&\subset&M_N(P^{N-1}_\mathbb C)
\end{matrix}$$
where $X_N,Y_N$ consist of bistochastic/stochastic matrices, and $K_N$ is the lift of $X_N$.
\end{proposition}

\begin{proof}
This follows from the above discussion. Indeed, the quotient map $S^{N-1}_\mathbb C\to P^{N-1}_\mathbb C$ induces the quotient map $M_N(S^{N-1}_\mathbb C)\to M_N(P^{N-1}_\mathbb C)$ at right, and the lift of the space of stochastic matrices $Y_N\subset M_N(P^{N-1}_\mathbb C)$ is then the rescaled group $U_N^N$, as claimed.
\end{proof}

In order to get some insight into the structure of $X_N,K_N$, we use inspiration from the Sinkhorn algorithm \cite{sin}, \cite{skn}. This algorithm starts with a $N\times N$ matrix having positive entries and produces, via successive averagings over rows/columns, a bistochastic matrix. In our situation, we would like to have an ``averaging'' map $Y_N\to Y_N$, whose infinite iteration lands in the model space $X_N$. Equivalently, we would like to have an ``averaging'' map $U_N^N\to U_N^N$, whose infinite iteration lands in $K_N$.

In order to construct such averaging maps, we use the orthogonalization procedure coming from the polar decomposition. First, we have the following result:

\begin{proposition}
We have orthogonalization maps as follows,
$$\xymatrix@R=10mm@C=15mm{
(S^{N-1}_\mathbb C)^N\ar[r]^\alpha\ar[d]&(S^{N-1}_\mathbb C)^N\ar[d]\\
(P^{N-1}_\mathbb C)^N\ar[r]^\beta&(P^{N-1}_\mathbb C)^N}$$
where $\alpha(x)_i=Pol([(x_i)_j]_{ij})$, and $\beta(p)=(P^{-1/2}p_iP^{-1/2})_i$, with $P=\sum_ip_i$.
\end{proposition}

\begin{proof}
Our first claim is that we have a factorization as in the statement. Indeed, pick $p_1,\ldots,p_N\in P^{N-1}_\mathbb C$, and write $p_i=Proj(x_i)$, with $||x_i||=1$. We can then apply $\alpha$, as to obtain a vector $\alpha(x)=(x_i')_i$, and then set $\beta(p)=(p_i')$, where $p_i'=Proj(x_i')$.

Our first task is to prove that $\beta$ is well-defined. Consider indeed vectors $\tilde{x}_i$, satisfying $Proj(\tilde{x}_i)=Proj(x_i)$. We have then $\widetilde{x}_i=\lambda_ix_i$, for certain scalars $\lambda_i\in\mathbb T$, and so the matrix formed by these vectors is $\widetilde{M}=\Lambda M$, with $\Lambda=diag(\lambda_i)$. It follows that $Pol(\widetilde{M})=\Lambda Pol(M)$, and so $\tilde{x}_i'=\lambda_ix_i$, and finally $Proj(\tilde{x}_i')=Proj(x_i')$, as desired.

It remains to prove that $\beta$ is given by the formula in the statement. For this purpose, observe first that, given $x_1,\ldots,x_N\in S^{N-1}_\mathbb C$, with $p_i=Proj(x_i)$ we have:
$$\sum_ip_i=\sum_i[(\bar{x}_i)_k(x_i)_l]_{kl}=\sum_i(\bar{M}_{ik}M_{il})_{kl}=((M^*M)_{kl})_{kl}=M^*M$$

We can now compute the projections $p_i'=Proj(x_i')$. Indeed, the coefficients  of these projections are given by $(p_i')_{kl}=\bar{U}_{ik}U_{il}$ with $U=MP^{-1/2}$, and we obtain, as desired:
\begin{eqnarray*}
(p_i')_{kl}
&=&\sum_{ab}\bar{M}_{ia}P^{-1/2}_{ak}M_{ib}P^{-1/2}_{bl}
=\sum_{ab}P^{-1/2}_{ka}\bar{M}_{ia}M_{ib}P^{-1/2}_{bl}\\
&=&\sum_{ab}P^{-1/2}_{ka}(p_i)_{ab}P^{-1/2}_{bl}
=(P^{-1/2}p_iP^{-1/2})_{kl}
\end{eqnarray*}

An alternative proof uses the fact that the elements $p_i'=P^{-1/2}p_iP^{-1/2}$ are self-adjoint, and sum up to 1. The fact that these elements are indeed idempotents can be checked directly, via $p_iP^{-1}p_i=p_i$, because this equality holds on $\ker p_i$, and also on $x_i$.
\end{proof}

As an illustration, here is how the orthogonalization works at $N=2$:

\begin{proposition}
At $N=2$ the orthogonalization procedure for $(Proj(x),Proj(y))$ amounts in considering the vectors $(x\pm y)/\sqrt{2}$, and then rotating by $45^\circ$.
\end{proposition}

\begin{proof}
By performing a rotation, we can restrict attention to the case $x=(\cos t,\sin t)$ and $y=(\cos t,-\sin t)$, with $t\in(0,\pi/2)$. Here the computations are as follows:
\begin{eqnarray*}
M=\begin{pmatrix}\cos t&\sin t\\ \cos t&-\sin t\end{pmatrix}
&\implies&
P=M^*M=\begin{pmatrix}2\cos^2t&0\\0&2\sin^2t\end{pmatrix}\\
&\implies&P^{-1/2}=|M|^{-1}=\frac{1}{\sqrt{2}}\begin{pmatrix}\frac{1}{\cos t}&0\\0&\frac{1}{\sin t}\end{pmatrix}\\
&\implies&U=M|M|^{-1}=\frac{1}{\sqrt{2}}\begin{pmatrix}1&1\\1&-1\end{pmatrix}
\end{eqnarray*}

Thus the orthogonalization procedure replaces $(Proj(x),Proj(y))$ by the orthogonal projections on the vectors $(\frac{1}{\sqrt{2}}(1,1),\frac{1}{\sqrt{2}}(-1,1))$, and this gives the result.
\end{proof}

With these preliminaries in hand, let us discuss now the version that we need of the Sinkhorn algorithm. The orthogonalization procedure is as follows:

\begin{proposition}
The orthogonalization maps $\alpha,\beta$ induce maps as follows,
$$\xymatrix@R=10mm@C=15mm{
U_N^N\ar[r]^\Phi\ar[d]&U_N^N\ar[d]\\
Y_N\ar[r]^\Psi&Y_N}$$
which are the transposition maps on $K_N,X_N$, and which are projections at $N=2$.
\end{proposition}

\begin{proof}
It follows from definitions that $\Phi(x)$ is obtained by putting the components of $x=(x_i)$ in a row, then picking the $j$-th column vectors of each $x_i$, calling $M_j$ this matrix, then taking the polar part $x_j'=Pol(M_j)$, and finally setting $\Phi(x)=x'$. Thus:
$$\Phi(x)=Pol((x_{ij})_i)_j\quad,\quad\Psi(u)=(P_i^{-1/2}u_{ji}P_i^{-1/2})_{ij}$$

Thus, the first assertion is clear, and the second assertion is clear too.
\end{proof}

At $N=3$ now, the algorithm doesn't stop any longer after 1 step. We obtain, after an infinite iteration, one of the 2 possible magic matrices coming from Latin squares.

Our first claim is that the algorithm converges, as follows:

\begin{conjecture}
The maps $\Phi,\Psi$ increase the volume,
$$vol:U_N^N\to Y_N\to[0,1],\quad vol(u)=\prod_j|\det((u_{ij})_i)|$$
and respectively land, after an infinite number of steps, in $K_N/X_N$.
\end{conjecture}

Observe that the quantities of type $|\det(p_1,\ldots,p_N)|$ are indeed well-defined, for any $p_1,\ldots,p_N\in P^{N-1}_\mathbb C$, because multiplying by scalars $\lambda_i\in\mathbb T$ doesn't change the volume. Thus, the volume map $vol:U_N^N\to[0,1]$ factorizes through $Y_N$, as stated above.

As a main application of the above conjecture, the infinite iteration $(\Phi^2)^\infty:U_N^N\to K_N$ would provide us with an integration on $K_N$, and hence on the quotient space $K_N\to X_N$ as well, by taking the push-forward measures, coming from the Haar measure on $U_N^N$. 

In relation now with the matrix model problematics, we have:

\begin{conjecture}
The universal $N\times N$ flat matrix representation
$$\pi_N:C(S_N^+)\to M_N(C(X_N)),\quad \pi_N(w_{ij})=(u\to u_{ij})$$
is faithful at $N=4$, and is inner faithful at any $N\geq5$.
\end{conjecture}

Regarding the $N=4$ conjecture, the problem is that of proving, as in \cite{bco}, that the composition $C(S_4^+)\to M_4(C(X_4))\to\mathbb C$ equals the Haar integration on $S_4^+$.

Regarding the $N\geq 5$ conjecture, the problem here is that of proving that the truncated moments $c_p^r$ in Proposition 3.1 converge with $r\to\infty$ to the Catalan numbers.

\section{Linear algebra}

Our purpose here is to advance towards a unification of the two conjectures formulated in section 5 above. The point indeed is that when trying to approach Conjecture 5.7 with the probabilistic tools coming from Proposition 3.1, the estimates that are needed seem to be related to those required for approaching Conjecture 5.6.

We first have the following definition, inspired from Proposition 3.1:

\begin{definition}
Associated to $x\in M_N(S^{N-1}_\mathbb C)$ is the $N^p\times N^p$ matrix
$$(T_p^x)_{i_1\ldots i_p,j_1\ldots j_p}=\frac{1}{N}<x_{i_1j_1},x_{i_2j_2}><x_{i_2j_2},x_{i_3j_3}>\ldots\ldots<x_{i_pj_p},x_{i_1j_1}>$$
where the scalar products are the usual ones on $S^{N-1}_\mathbb C\subset\mathbb C^N$.
\end{definition}

The first few values of these matrices, at $p=1,2,3$, are as follows:
\begin{eqnarray*}
(T_1^x)_{ia}&=&\frac{1}{N}<x_{ia},x_{ia}>=\frac{1}{N}\\
(T_2^x)_{ij,ab}&=&\frac{1}{N}<x_{ia},x_{jb}><x_{jb},x_{ia}>=\frac{1}{N}|<x_{ia},x_{jb}>|^2\\
(T_3^x)_{ijk,abc}&=&\frac{1}{N}<x_{ia},x_{jb}><x_{jb},x_{kc}><x_{kc},x_{ia}>
\end{eqnarray*}

The interest in these matrices, in connection with Conjecture 5.7, comes from:

\begin{proposition}
For the universal model, the matrices $T_p$ in Proposition 3.1 are
$$T_p=\int_{K_N}T_p^xdx\,$$
where $dx$ is the measure on the model space $K_N$ coming from Conjecture 5.6.
\end{proposition}

\begin{proof}
This is a trivial statement, because by definition of $T_p$, we have:
\begin{eqnarray*}
(T_p)_{i_1\ldots i_p,j_1\ldots j_p}
&=&tr(u_{i_1j_1}\ldots u_{i_pj_p})=\int_{K_N}tr(u_{i_1j_1}^x\ldots u_{i_pj_p}^x)dx\\
&=&\int_{K_N}tr(Proj(x_{i_1j_1})\ldots Proj(x_{i_px_p}))dx\\
&=&\frac{1}{N}\int_{K_N}<x_{i_1j_1},x_{i_2j_2}>\ldots\ldots<x_{i_pj_p},x_{i_1j_1}>dx\\
&=&\int_{K_N}(T_p^x)_{i_1\ldots i_p,j_1\ldots j_p}dx
\end{eqnarray*}

Thus the formula in the statement holds indeed.
\end{proof}

Our claim is that the matrices $T_p^x$ are related to Conjecture 5.6 as well. To any noncrossing partition $\pi\in NC(1,\ldots,p)$ let us associate the following vector of $(\mathbb C^N)^{\otimes p}$: 
$$\xi_\pi=\sum_{\ker i\leq\pi}e_{i_1}\otimes\ldots\otimes e_{i_p}$$

These vectors appear in the representation theory of $S_N^+$. See \cite{bco}.

At $p=1$, we obtain the 1-eigenvector of $T_1^x=(1/N)_{ij}$:
$$\xi_|=\sum_ie_i$$

At $p=2$ now, the two vectors constructed above are as follows:
$$\xi_{||}=\sum_{ij}e_i\otimes e_j\quad,\quad \xi_\sqcap=\sum_ie_i\otimes e_i$$

In general, we have the following result:

\begin{proposition}
For any $x\in M_N(S^{N-1}_\mathbb C)$, the following hold:
\begin{enumerate}
\item If $\{x_{ij}\}_i$ are pairwise orthogonal then $(T_p^x)^*\xi_{||\ldots|}=\xi_{||\ldots|}$ and $T_p^x\xi_{\sqcap\hskip-0.5mm\sqcap\ldots\sqcap}=\xi_{\sqcap\hskip-0.5mm\sqcap\ldots\sqcap}$.

\item If $\{x_{ij}\}_j$ are pairwise orthogonal then $T_p^x\xi_{||\ldots|}=\xi_{||\ldots|}$ and $(T_p^x)^*\xi_{\sqcap\hskip-0.5mm\sqcap\ldots\sqcap}=\xi_{\sqcap\hskip-0.5mm\sqcap\ldots\sqcap}$.

\item If $\{x_{ij}\}_i$ or $\{x_{ij}\}_j$ are pairwise orthogonal then $<T_p^x\xi_{||\ldots|},\xi_{||\ldots|}>=N^p$.

\item We have $<T_p^x\xi_{\sqcap\hskip-0.5mm\sqcap\ldots\sqcap},\xi_{\sqcap\hskip-0.5mm\sqcap\ldots\sqcap}>=N$, without assumptions on $x$.
\end{enumerate}
\end{proposition}

\begin{proof}
It is elementary to see that we have $(T_p^x)^*=T_p^{x^*}$, and so it is enough to establish the assertions in (1,2) regarding the eigenvalues of $T_p^x$. The proof goes as follows:

(1) Assuming that $\{x_{ij}\}_i$ are pairwise orthogonal, we have indeed:
\begin{eqnarray*}
(T_p^x\xi_{\sqcap\hskip-0.5mm\sqcap\ldots\sqcap})_{i_1\ldots i_p}
&=&\sum_j(T_p^x)_{i_1\ldots i_p,j\ldots j}
=\frac{1}{N}\sum_j<x_{i_1j},x_{i_2j}>\ldots\ldots<x_{i_pj},x_{i_1j}>\\
&=&\frac{1}{N}\sum_j\delta_{i_1i_2}\ldots\delta_{i_pi_1}=\delta_{i_1,\ldots,i_p}
\end{eqnarray*}

(2) Assuming now that $\{x_{ij}\}_j$ are pairwise orthogonal, we have indeed:
$$(T_p^x\xi_{||\ldots|})_{i_1\ldots i_p}
=\sum_{j_1\ldots j_p}(T_p^x)_{i_1\ldots i_p,j_1\ldots j_p}
=\frac{1}{N}\sum_{j_1\ldots j_p}<x_{i_1j_1},x_{i_2j_2}>\ldots<x_{i_pj_p},x_{i_1j_1}>
=1$$

Here we have used, $p$ times via a recurrence, the fact that given an orthonormal basis $\{e_k\}$ we have $\sum_k<x,e_k><e_k,y>=<x,y>$, for any two vectors $x,y$.

(3) The scalar product in the statement is given by:
$$<T_p^x\xi_{||\ldots|},\xi_{||\ldots|}>
=\sum_{i_1\ldots i_p,j_1\ldots j_p}(T_p^x)_{i_1\ldots i_p,j_1\ldots j_p}\\
=\sum_{i_1\ldots i_p}(T_p^x\xi_{||\ldots|})_{i_1\ldots i_p}$$

When $\{x_{ij}\}_j$ are pairwise orthogonal, by using (2) we obtain $N^p$, as claimed. Since $(T_p^x)^*=T_p^{x^*}$, the result follows to hold when $\{x_{ij}\}_i$ are pairwise orthogonal too.

(4) We have the following computation, valid for any $x$:
\begin{eqnarray*}
<T_p^x\xi_{\sqcap\hskip-0.5mm\sqcap\ldots\sqcap},\xi_{\sqcap\hskip-0.5mm\sqcap\ldots\sqcap}>
&=&\sum_i(T_p^x\xi_{\sqcap\hskip-0.5mm\sqcap\ldots\sqcap})_{i\ldots i}
=\sum_{ij}(T_p^x)_{i\ldots i,j\ldots j}\\
&=&\frac{1}{N}\sum_{ij}<x_{ij},x_{ij}>^p
=N
\end{eqnarray*}

But this proves the last assertion, and we are done.
\end{proof}

The above computations suggest the following definition:

\begin{definition}
Associated to any $x\in M_N(S^{N-1}_\mathbb C)$ is the function
$$F_p(x)=\frac{1}{N^p}||T_p^x\xi_{\sqcap\hskip-0.5mm\sqcap\ldots\sqcap}||^2$$
depending on a fixed integer $p\geq2$.
\end{definition}

Observe that, according to the formula of $T_p^x$, we have:
$$F_p(x)=\frac{1}{N^{p+2}}\sum_{i_1\ldots i_p}\left|\sum_j<x_{i_1j},x_{i_2j}>\ldots\ldots<x_{i_pj},x_{i_1j}>\right|^2$$

We have the following statement, supported by computer calculations:

\begin{conjecture}
For any $x\in U_N^N$, and any $p\geq2$, we have 
$$F_p(x)\geq F_p(\Psi^2(x))$$
with equality iff $x\in K_N$, in which case $F_p(x)=1$. 
\end{conjecture}

By a compacity argument, this would prove that our Sinkhorn type algorithm converges. Thus, we have here a first step towards unifying Conjecture 5.6 and Conjecture 5.7.

Let us restrict now attention to the case $p=2$. Here we have:
$$F_2(x)=\frac{1}{N^4}\sum_{ij}\left(\sum_k|<x_{ik},x_{jk}>|^2\right)^2$$

At $N=2$, by writing the inequality in Conjecture 6.5 in terms of the orthogonal projections $P,Q,R,S$ on the vectors $x_{ij}$, we are led to the following statement: 

\begin{conjecture}
Let $P,Q,R,S\in M_K(\mathbb C)$ be orthogonal projections satisfying:
\begin{enumerate}
\item $P\perp Q$.

\item $R\perp S$.

\item $Im(P)\cap Im(R)=\{0\}$.

\item $Im(Q)\cap Im(S)=\{0\}$.

\item $rank(P)+rank(Q)=rank(R)+rank(S)$.

\item $rank(P)+rank(R)=rank(Q)+rank(S)$.

\end{enumerate}
We have then the following inequality,
$$Tr(PR)+Tr(QS)\geq Tr(P'Q')+Tr(R'S')$$
where $P',Q',R',S'$ are the following orthogonal projections
\begin{eqnarray*}
P'&=&(P+R)^{-1/2}P(P+R)^{-1/2}\\
Q'&=&(Q+S)^{-1/2}Q(Q+S)^{-1/2}\\
R'&=&(P+R)^{-1/2}R(P+R)^{-1/2}\\
S'&=&(Q+S)^{-1/2}S(Q+S)^{-1/2}
\end{eqnarray*}
with all the inverses taken in the sense of Moore-Penrose.
\end{conjecture}

We only know how to prove a special case of the statement above:

\begin{proposition}
Conjecture 6.6 holds for $S=0$.
\end{proposition}

\begin{proof}
We can write $P,R$ by using the Halmos normal form \cite{hal}: 
\begin{eqnarray*}
P&=&I_{00}\oplus I_{01}\oplus 0_{10}\oplus 0_{11}\oplus U^*\begin{pmatrix}I&0\\0&0\end{pmatrix}U\\
R&=&I_{00}\oplus 0_{01}\oplus I_{10}\oplus 0_{11}\oplus U^*\begin{pmatrix}I-H &W\\W&H\end{pmatrix}U
\end{eqnarray*}

By using the condition (3) in the statement, we can replace the first term in the direct sums above by $0$. Now by using the fact that $H,W$ commute, we have:
\begin{eqnarray*}
P'&=&0_{00}\oplus I_{01}\oplus 0_{10}\oplus 0_{11}\oplus\frac{1}{2}U^*\begin{pmatrix}I+\sqrt{H}&-\sqrt{I-H}\\-\sqrt{I-H}&I-\sqrt{H}\end{pmatrix}U\\
R'&=&0_{00}\oplus 0_{01}\oplus I_{10}\oplus 0_{11}\oplus\frac{1}{2}U^* \begin{pmatrix}I-\sqrt{H}&\sqrt{I-H}\\ \sqrt{I-H}&I+\sqrt{H}\end{pmatrix}U
\end{eqnarray*}

We therefore have the following estimate:
$$Tr(PR)=Tr(I-H)\geq Tr(I-\sqrt{H})=2Tr(P'(I-P))\geq 2 Tr(P'Q)$$

Thus we have obtained the desired inequality.
\end{proof}

\end{document}